\documentclass[a4paper,12pt]{article}
\usepackage[english]{babel}
\usepackage{mathtext}
\usepackage[T2A]{fontenc}
\usepackage{inputenc,amssymb}
\usepackage{epsfig}
\usepackage{caption2}
\usepackage{amsfonts}
\usepackage{amssymb}
\usepackage{amsmath}
\usepackage{amsthm}
\usepackage{enumerate}
\usepackage{geometry}
\usepackage{color}
\usepackage{hyperref}
\hypersetup{
colorlinks=true,
linkcolor=blue,
citecolor=blue,
urlcolor=black
}

\begin{document}

\renewcommand{\r}{\mathbb R}
\newcommand{\eqd}{\stackrel{d}{=}}
\renewcommand{\figurename}{{\bf Figure}}
\renewcommand{\tablename}{{\bf Table}}
\newtheorem{theorem}{Theorem}
\newtheorem{lemma}{Lemma}
\newtheorem{corollary}{Corollary}
\newtheorem{proposition}{Proposition}
\renewcommand*{\proofname}{\bf Proof}

\title{\vspace{-1.8cm}
Improved mathematical models of statistical regularities in precipitation}

\author{
V.\,Yu.~Korolev\textsuperscript{1},
A.\,K.~Gorshenin\textsuperscript{2}
}

\date{}

\maketitle

\footnotetext[1]{Faculty of Computational Mathematics and
Cybernetics, Lomonosov Moscow State University, Russia; Institute of
Informatics Problems, Federal Research Center ``Computer Science and
Control'' of Russian Academy of Sciences, Russia; Hangzhou Dianzi
University, China; \url{vkorolev@cs.msu.su}}

\footnotetext[2]{Institute of Informatics Problems, Federal Research
Center ``Computer Science and Control'' of Russian Academy of
Sciences, Russia; Faculty of Computational Mathematics and
Cybernetics, Lomonosov Moscow State University, Russia; \url{agorshenin@frccsc.ru}}

{\bf Abstract.} The paper presents improved mathematical models and methods for statistical regularities in the behavior of some important characteristics of precipitation: duration of a wet period, maximum daily and total precipitation volumes within a such period. The asymptotic approximations are deduced using limit theorems for statistics constructed from samples with random sizes having the generalized negative binomial (GNB) distribution. It demonstrates excellent concordance with the empirical distribution of the duration of wet periods measured in days. The asymptotic distribution of the maximum daily precipitation volume within a wet period turns out to be a tempered scale mixture of the gamma distribution with the scale factor having the Weibull distribution, whereas the asymptotic approximation to the total precipitation volume for a wet period turns out to be the generalized gamma (GG) distribution. Two approaches to the definition of abnormally extremal precipitation are presented. The first approach is based on an excess of a certain quantile of the asymptotic distribution of the maximum daily precipitation. The second approach is based on the GG model for the total precipitation volume. The corresponding statistical test is compared with a previously proposed one based on thå classical gamma distribution using real precipitation data.

\smallskip


\section{Introduction}

In this paper improved mathematical models and methods for
statistical regularities in the behavior of such characteristics of
precipitation as the duration of a wet period, maximum daily
precipitation within a wet period and total precipitation volume per
a wet period are proposed. The importance of studying such objects for climate problems has been described, for example, in~\cite{Lockhoff2014,Zolinaetal2005,Zolinaetal2009,Zolina2013,Zolina2014}. The base for the improved models (comparing of those proposed in~\cite{Korolevetal2018}) is the generalized negative
binomial (GNB) distribution. The results of fitting the GNB
distribution to real data are presented and demonstrate excellent
concordance of the GNB model with the empirical distribution of the
duration of wet periods measured in days. Based on this GNB model,
asymptotic approximations are proposed for the distributions of the
maximum daily precipitation volume within a wet period and of the
total precipitation volume for a wet period. The asymptotic
distribution of the maximum daily precipitation volume within a wet
period turns out to be a tempered scale mixture of the gamma
distribution in which the scale factor has the Weibull distribution,
whereas the asymptotic approximation for the total precipitation
volume for a wet period turns out to be the generalized gamma (GG)
distribution. Both approximations appear to be very accurate. These
asymptotic approximations are deduced using limit theorems for
statistics constructed from samples with random sizes having the
generalized negative binomial distribution.

A rather reasonable approach to the unambiguous
(algorithmic) determination of extreme or abnormally heavy total
precipitation for a wet period~\cite{Korolevetal2018} is realized with a GNB
model for the duration of wet periods measured in days. This model
is well justified statistically and theoretically by means of
special limit theorems of probability theory which yield an
asymptotic approximation to the distribution of the total
precipitation volume within a wet period. This approximation has the
form of a generalized gamma (GG) distribution. The proof of this
result is based on the law of large numbers for random sums in which
the number of summands has the GNB distribution. Hence, the
hypothesis that the total precipitation volume during a certain wet
period is abnormally large is re-formulated as the homogeneity
hypothesis of a sample from the GG distribution. 

The paper is organized as follows. In Section~\ref{SecDefinitions} necessary definitions and auxiliary results are given. GNB model for fitting the duration of wet periods is introduced in Section~\ref{SecGNB}. In Section~\ref{SecApprox1} the theorems about asymptotic probability distribution of extremal daily precipitation within a wet period and its properties are proved. A statistical method for estimating stability parameters of
daily precipitation trends is demonstrated in Section~\ref{SecApprox2}. Also, the generalization of the R{\'e}nyi theorem for GNB random sums is proved. In Section~\ref{SecComparison} an extremality statistical test based on assumptions about GG distribution of precipitation volumes is introduced. Also, the corresponding results are compared with the decisions of gamma distribution based test using real precipitation data for Potsdam and Elista. Section~\ref{SecConclusion} is devoted to the main conclusions of the work.

\section{Definitions and auxiliary results}
\label{SecDefinitions}

A GG distribution is the absolutely continuous distribution defined
by the density
\begin{equation}
\label{GGdensity}
g^*(x;r,\gamma,\mu)=\frac{|\gamma|\mu^r}{\Gamma(r)}x^{\gamma
r-1}e^{-\mu x^{\gamma}},\ \ \ \ x\ge0,
\end{equation}
with $\gamma\in\mathbb{R}$, $\mu>0$, $r>0$.
A random value with the density $g^*(x;r,\gamma,\mu)$ will
be denoted $\overline{G}_{r,\gamma,\mu}$. The 
GG distributions were first described~\cite{Stacy1962} as a unitary family of probability distributions simultaneously containing both Weibull and gamma distributions. 
A random value having the gamma distribution with shape parameter $r>0$ and scale parameter $\mu>0$ will be denoted $G_{r,\mu}\eqd
\overline{G}_{r,1,\mu}$ (here symbol $\eqd$ denotes the coincidence of distributions).

A random value $W_{\gamma}$ with the {\it Weibull distribution}  is a particular case of GG distributions corresponding to the density $g^*(x;1,\gamma,1)$
with $\gamma>0$, so, $W_{\gamma}\eqd
\overline{G}_{1,\gamma,1}$.

Let $r>0$, $\gamma\in\r$ and $\mu>0$.The random value $N_{r,\gamma,\mu}$ has the {\it generalized negative binomial
$($GNB$)$ distribution}, if
\begin{equation}
\label{GNB}
{\mathbb P}(N_{r,\gamma,\mu}=k)=\frac{1}{k!}\int_{0}^{\infty}e^{-z}z^kg^*(z;r,\gamma,\mu)dz,\
\ \ \ k=0,1,2...,
\end{equation}
where $g^*(z;r,\gamma,\mu)$ is determined by the formula~\eqref{GGdensity}.

The following asymptotic property of the GNB distribution~\cite{KorolevZeifman2017} will play the fundamental role in the construction of asymptotic
approximations to the distributions of extreme daily precipitation
within a wet period and the total precipitation volume per a wet
period and the corresponding statistical tests for precipitation to
be abnormally heavy.

\begin{lemma}
\label{Lemma1}
For $r>0$,
$\gamma\in\mathbb{R}$, $\mu>0$ let $N_{r,\gamma,\mu}$ be a random value with
the GNB distribution. We have
\begin{equation}\label{15}
\mu^{1/\gamma}N_{r,\gamma,\mu}\Longrightarrow
\overline{G}_{r,\gamma,1}\eqd G_{r,1}^{1/\gamma}
\end{equation}
as $\mu\to0$. If, moreover, $r\in(0,1]$ and $\gamma\in(0,1]$, then
the limit law can be represented as
\begin{gather}
\overline{G}_{r,\gamma,1}\eqd
\frac{W_1}{S_{\gamma,1}Z_{r,1}^{1/\gamma}}\eqd\frac{W_1^{1/\gamma}}{Z_{r,1}^{1/\gamma}}\eqd\notag\\
\eqd
\bigg(\frac{W_1G_{r,1}}{G_{r,1}+G_{1-r,1}}\bigg)^{1/\gamma}\eqd
W_1^{1/\gamma}\cdot\big(1+{\textstyle\frac{1-r}{r}}Q_{1-r,r}\big)^{-1/\gamma},
\label{16}
\end{gather}
where the random values $W_1$, $S_{\gamma,1}$ and $Z_{r,1}$ are independent
as well as the random values $W_1$ and $Z_{r,1}$, or the random values $W_1$,
$G_{r,1}$ and $G_{1-r,1}$, and the random value $Q_{1-r,r}$ has the
Snedecor--Fisher distribution with parameters $1-r$ and $r$. 
\end{lemma}

Let $\lambda>0$, $\gamma>0$. Instead of an infinitesimal parameter
$\mu$, in order to construct asymptotic approximations with ``large''
sample size, introduce an auxiliary ``infinitely large'' parameter
$n\in\mathbb{N}$ and assume that $\mu=\mu_n=\lambda n^{-\gamma}$.
Then from Lemma~\ref{Lemma1} and \eqref{15}, it follows that for $r>0$, $\mu>0$ we have
\begin{equation}\label{2}
n^{-1}G_{r,\gamma,\lambda/n^{\gamma}}\Longrightarrow
\overline{G}_{r,\gamma,\lambda}\eqd\lambda^{-1/\gamma}\overline{G}_{r,\gamma,1}\eqd\lambda^{-1/\gamma}G_{r,1}^{1/\gamma}
\end{equation}
as $n\to\infty$.

\begin{lemma}
\label{Lemma2}
Let $\Lambda_1,\Lambda_2,\ldots$ be a sequence
of positive random values such that for any $n\in\mathbb{N}$ the random value
$\Lambda_n$ is independent of the Poisson process $P(t)$, $t\ge0$.
The convergence
\begin{equation*}
n^{-1}P(\Lambda_n)\Longrightarrow \Lambda
\end{equation*}
as $n\to\infty$ to some nonnegative random value $\Lambda$ takes place if
and only if
\begin{equation}
n^{-1}\Lambda_n\Longrightarrow \Lambda \label{3}
\end{equation}
as $n\to\infty$.
\end{lemma}

This statement is a particular case of Lemma~2 in~\cite{Korolev1998}.

Consider a sequence of independent identically distributed (i.i.d.)
random values $X_1,X_2,\ldots$. Let $N_1,N_2,\ldots$ be a sequence of
natural-valued random values such that for each $n\in\mathbb{N}$ the random value
$N_n$ is independent of the sequence $X_1,X_2,\ldots$. Denote
$M_n=\max\{X_1,\ldots,X_{N_n}\}$.

\begin{lemma}
\label{Lemma3}
Let $\Lambda_1,\Lambda_2,\ldots$ be a sequence
of positive random values such that for each $n\in\mathbb{N}$ the random value
$\Lambda_n$ is independent of the Poisson process $P(t)$, $t\ge0$.
Let $N_n=P(\Lambda_n)$. Assume that there exists a nonnegative random value
$\Lambda$ such that convergence~{\rm \eqref{3}} takes place. Let
$X_1,X_2,\ldots$ be i.i.d. random values with a common d.f. $F(x)$. Assume
also that $\sup\{x:\,F(x)<1\}=\infty$ and there exists a number
$\alpha>0$ such that for each $x>0$
\begin{equation}
\lim_{y\to\infty}\frac{1-F(xy)}{1-F(y)}=x^{-\alpha}.\label{4}
\end{equation}
Then
\begin{equation*}
\lim_{n\to\infty}\sup_{x\ge 0}\bigg|{\sf
P}\bigg(\frac{M_n}{F^{-1}(1-\frac{1}{n})}<x\bigg)-\int_{0}^{\infty}e^{-zx^{-\alpha}}d{\sf
P}(\Lambda<z)\bigg|=0.
\end{equation*}
\end{lemma}

This statement is a particular case of Theorem 3.1 in~\cite{KorolevSokolov2008}.

Consider a sequence of random values $W_1,W_2,...$ Let $N_1,N_2,...$ be
natural-valued random values such that for every $n\in\mathbb{N}$ the random value
$N_n$ is independent of the sequence $W_1,W_2,...$ In the following
statement~\cite{Korolev1994,Korolev1995} the convergence is meant as $n\to\infty$.

\begin{lemma}
\label{Lemma4}
Assume that
there exist an infinitely increasing $($convergent to zero$)$
sequence of positive numbers $\{b_n\}_{n\ge1}$ and a random value $W$ such
that
\begin{equation*}
b_n^{-1}W_n\Longrightarrow W.
\end{equation*}
If there exist an infinitely increasing $($convergent to zero$)$
sequence of positive numbers $\{d_n\}_{n\ge1}$ and a random value $N$ such
that
\begin{equation}\label{T31}
d_n^{-1}b_{N_n}\Longrightarrow N,
\end{equation}
then
\begin{equation}\label{T32}
d_n^{-1}W_{N_n}\Longrightarrow W\cdot N,
\end{equation}
where the random values on the right-hand side of {\rm\eqref{T32}} are
independent. If, in addition, $N_n\longrightarrow\infty$ in
probability and the family of scale mixtures of the d.f. of the random value
$W$ is identifiable, then condition {\rm\eqref{T31}} is not only
sufficient for {\rm\eqref{T32}}, but is necessary as well.
\end{lemma}

\section{Generalized negative binomial model for the duration of wet periods}
\label{SecGNB}

It turned out that the statistical regularities of the number of
subsequent wet days can be very reliably modeled by the negative
binomial distribution with the shape parameter less than one~\cite{Gorshenin2017a,Gulev}. The analytic and asymptotic properties of the GNB distributions were
studied in~\cite{KorolevZeifman2017}. Since the GG distribution
is a more general and hence, more flexible model than the ``pure''
gamma-distribution, there arises a hope that the GNB distribution
could provide even better goodness of fit to the statistical
regularities in the duration of wet periods than the ``pure''
negative binomial distribution. Negative binomial distributions are special cases of the GNB distributions (the parameter $\gamma$ in Eq.~\eqref{GNB} should be equal $1$).

On Figs.~\ref{FigGNBl1Potsdam} and~\ref{FigGNBl1Elista} there are the histograms constructed from real data of $3323$ wet periods in Potsdam and $2937$ wet periods in
Elista during almost $60$ years. On the same picture there are the graphs of the fitted
negative binomial distribution  and the fitted GNB distribution with additionally
adjusted scale and power parameters. For vividness, in the GNB model
the value of the shape parameter $r$ was taken the same as that
obtained for the NB model and equal to $0.876$ for Elista and $0.847$
for Potsdam. For ``fine tuning'' of the GNB models with these fixed
values of $r$ three procedures were used based on the minimization
of the distance between the histogram and the fitted GNB model~\cite{Gorshenin2018CCIS}:
\begin{itemize}
\item minimization of the $\ell_1$-distance;
\item minimization of the $\ell_2$-distance;
\item minimization of the $\ell_{\infty}$-distance (i. e., the uniform distance or the $\sup$-norm).
\end{itemize}

\begin{figure}[!h]
\centering
\includegraphics[width=\textwidth,height=0.3\textheight]{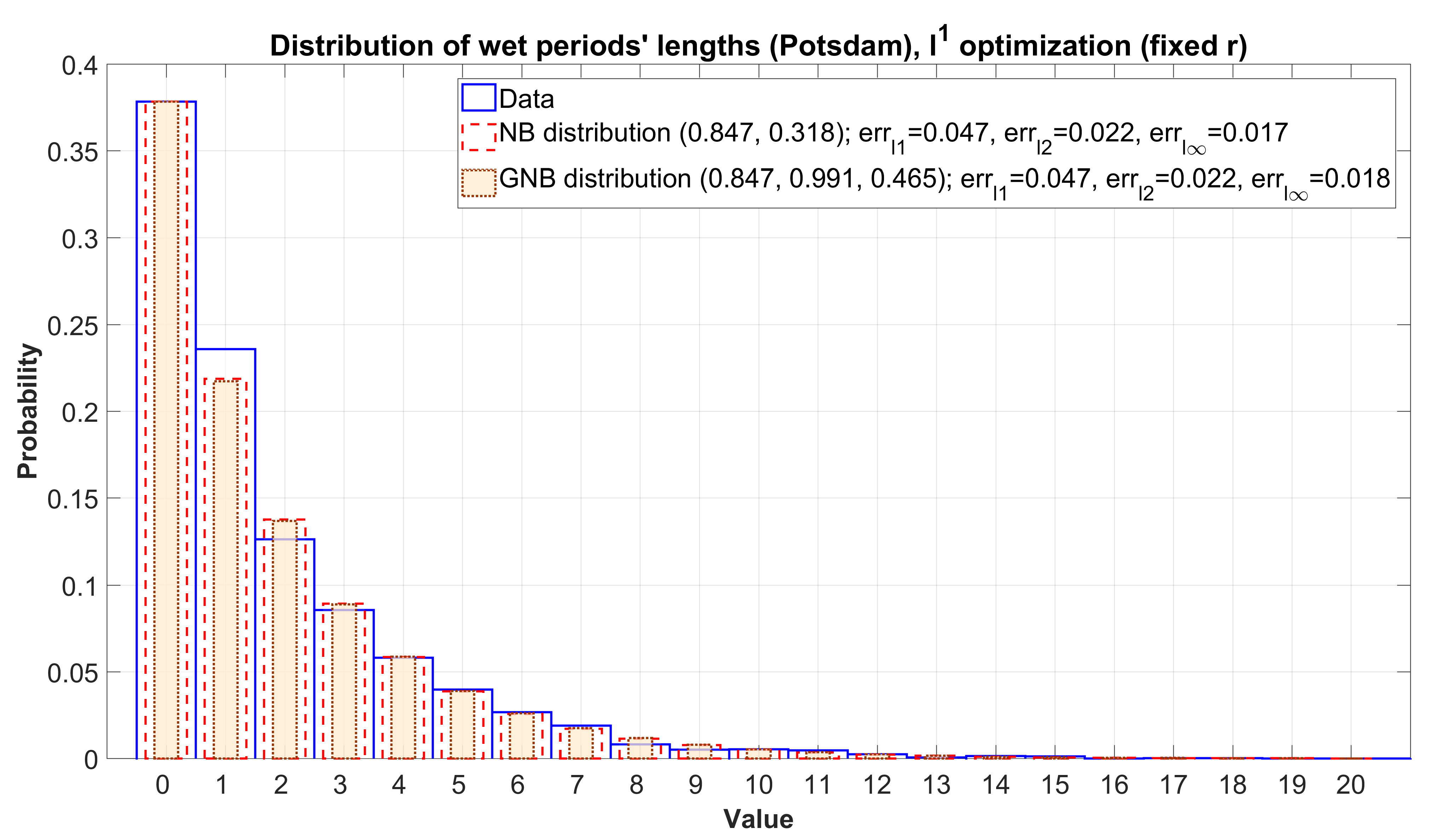} 
 \caption{The histograms constructed
from real data of 3320 wet
periods in Potsdam and the fitted NB and GNB models,
$\ell_1$-distance minimization.}\label{FigGNBl1Potsdam}
\end{figure}

On Figs.~\ref{FigGNBl1Potsdam} and~\ref{FigGNBl1Elista} the results of $\ell_1$-distance minimization are presented.  Other two procedures yield almost the same results. Minimization of
the $\ell_{\infty}$-distance results in that this distance between
the histogram and the GNB distribution becomes almost 4.5 times less
than the $\ell_{\infty}$-distance between the histogram and the NB
distribution. However, as this is so, the value of the
$\ell_{\infty}$-distance remain noticeably greater than those
obtained by the minimization of the $\ell_1$- and $\ell_2$-distances
which also give the GNB model an essential advantage over the NB model in the accuracy. But for Potsdam data
the advantage of the adjusted GNB model over the NB model is not so
crucial and does not exceed $10\%$. It should be especially noted that
the $P$-values of the chi-square goodness-of-fit test are
practically equal to 1 for all NB and GNB models mentioned above.
So, the choice of a proper metric to be minimized remains an option. More examples of functional estimation GNB and GG distributions and implemented MATLAB software solutions are presented in~\cite{Gorshenin2018a}.

\begin{figure}[!h]
\centering
\includegraphics[width=\textwidth,height=0.3\textheight]{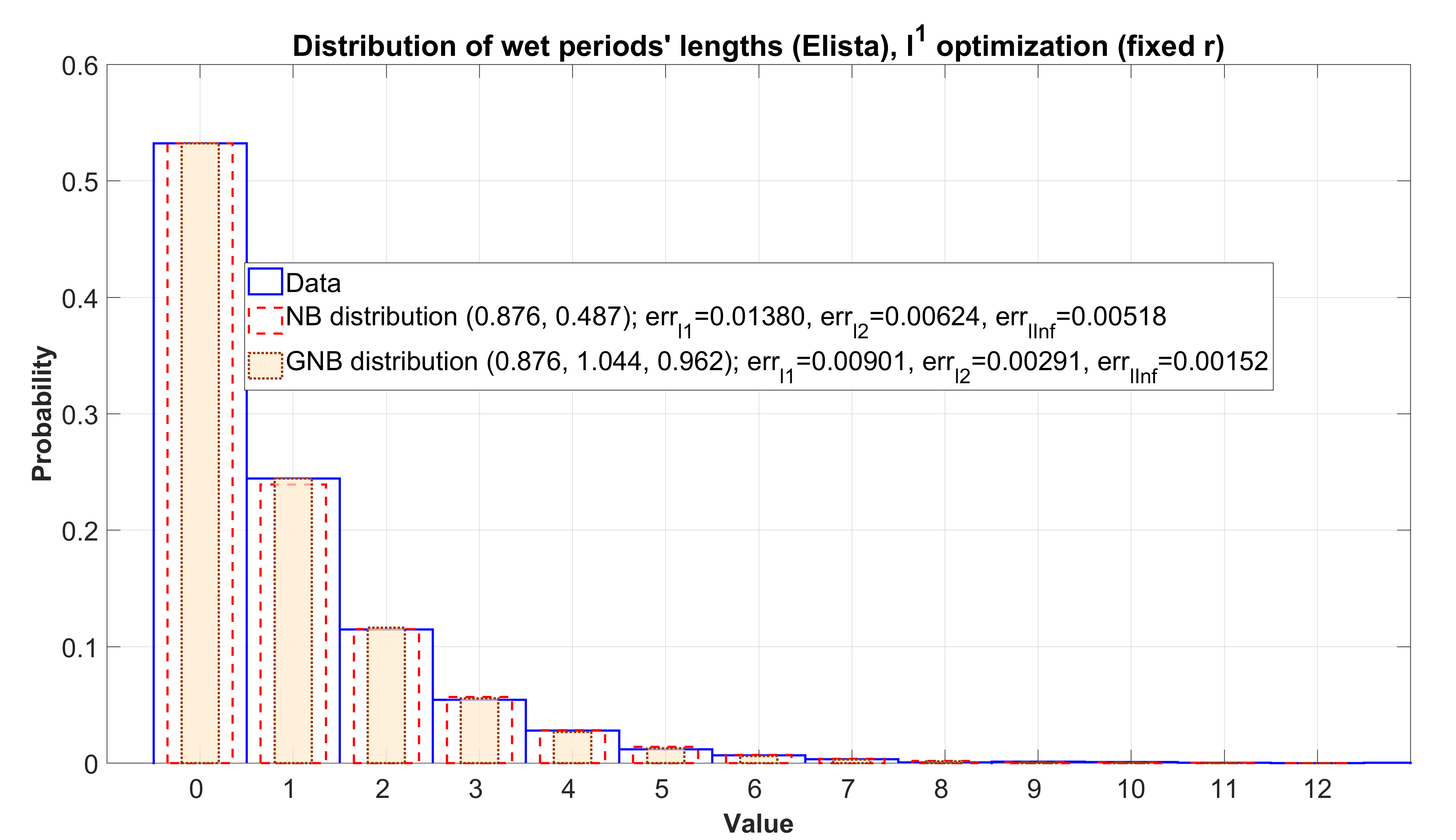}
 \caption{The histograms constructed
from real data of 2937 wet periods in Elista and the fitted NB and GNB models,
$\ell_1$-distance minimization.}\label{FigGNBl1Elista}
\end{figure}

\section{The asymptotic approximation to the probability\\ distribution of extremal daily precipitation within a wet period}
\label{SecApprox1}

In this section we will deduce the probability distribution of
extremal daily precipitation within a wet period.

\begin{theorem}
\label{Th1}
Let $n\in\mathbb{N}$, $\gamma>0$, $\lambda>0$
and let $N_{r,\gamma,\mu_n}$ be a random value with the GNB distribution
with parameters $r>0$, $\gamma>0$ and $\mu_n=\lambda/n^{\gamma}$.
Let $X_1,X_2,\ldots$ be i.i.d. random values with a common d.f. $F(x)$.
Assume that $\mathrm{rext}(F)=\infty$ and there exists a number
$\alpha>0$ such that relation~{\rm \eqref{4}} holds for any $x>0$.
Then
\begin{equation}\label{th2_1}
\lim_{n\to\infty}\sup_{x\ge 0}\bigg|{\sf
P}\bigg(\frac{\max\{X_1,\ldots,X_{N_{r,\gamma,\mu_n}}\}}{F^{-1}(1-\frac{1}{n})}
<x\bigg)-F(x; r,\alpha,\gamma,\lambda)\bigg|=0,
\end{equation}
where $F(x; r,\alpha,\gamma,\lambda)={\sf
P}(M_{r,\alpha,\gamma,\lambda}<x)$, $x\in\mathbb{R}$, 
\begin{equation}\label{M}
M_{r,\alpha,\gamma,\lambda}\eqd
\frac{\overline{G}_{r,\alpha\gamma,\lambda}}{W_{\alpha}}\eqd
\bigg(\frac{\overline{G}_{r,\gamma,\lambda}}{W_1}\bigg)^{1/\alpha}\eqd
\lambda^{-1/\alpha\gamma}\Big(\frac{G_{r,1}}{W_{\gamma}}\Big)^{1/\alpha\gamma}
\end{equation}
and in each term the involved random variables are
independent. 
\end{theorem}

\begin{proof}
It is well known that the negative binomial distribution is a mixed
Poisson distribution with the gamma mixing distribution~\cite{GreenwoodYule1920}.
So, $N_{r,\gamma,\mu_n}\eqd P\big(\overline{G}_{r,\gamma,\mu_n}\big)$.
Therefore, from~\eqref{2}, Lemma~\ref{Lemma2} with
$\Lambda_n=\overline{G}_{r,\gamma,\mu_n}$ and Lemma~\ref{Lemma3} with the
account of the absolute continuity of the limit distribution it
immediately follows that
\begin{equation*}
\lim_{n\to\infty}\sup_{x\ge 0}\bigg|{\sf
P}\bigg(\frac{\max\{X_1,\ldots,X_{N_{r,\gamma,\mu_n}}\}}{F^{-1}(1-\frac{1}{n})}<x\bigg)-
\int_{0}^{\infty}e^{-zx^{-\alpha}}g^*(z;r,\gamma,\lambda)dz\bigg|=0.
\end{equation*}

Since the Fr{\'e}chet (inverse Weibull) d.f. $e^{-x^{-\alpha}}$ with
$\alpha>0$ corresponds to the random value $W_{\alpha}^{-1}$, it is easy to
make sure
\begin{equation*}
F(x;r,\alpha,\gamma,\lambda)\equiv\int_{0}^{\infty}e^{-zx^{-\alpha}}g^*(z;r,\gamma,\lambda)dz={\sf
P}\bigg(\frac{\overline{G}_{r,\gamma,\lambda}^{1/\alpha}}{W_{\alpha}}<x\bigg).
\end{equation*}

Moreover, using relation $\overline{G}_{r,\gamma,\mu}\eqd G_{r,\mu}^{1/\gamma}$, it is easy to see that
\begin{equation*}
\frac{\overline{G}_{r,\gamma,\lambda}^{1/\alpha}}{W_{\alpha}}\eqd
\frac{\overline{G}_{r,\alpha\gamma,\lambda}}{W_{\alpha}}\eqd
\bigg(\frac{\overline{G}_{r,\gamma,\lambda}}{W_1}\bigg)^{1/\alpha}\eqd
\lambda^{-1/\alpha\gamma}\Big(\frac{G_{r,1}}{W_{\gamma}}\Big)^{1/\alpha\gamma}
\end{equation*}
where in each term the involved random variables are independent.
\end{proof}

It is worth noting that if $\gamma=1$, then the limit distribution $F(x;r,\alpha,1,\lambda)$ corresponds to the results of~\cite{KorolevGorshenin2017DAN}.

\begin{theorem}
The distribution of the random value $M_{r,\alpha,\gamma,\lambda}$ can be represented as follows.
\begin{enumerate} [(i)]
\item If  $r\in(0,1]$, it is the scale mixture of the distribution of the ratio of two independent Weibull-distributed random variables:
\begin{equation*}
M_{r,\alpha,\gamma,\lambda}\eqd \big(\lambda
Z_{r,1}\big)^{-1/\alpha\gamma}\cdot\frac{W_{\alpha\gamma}}{W_{\gamma}},
\end{equation*}

All the involved random variables are independent, and random value $Z_{r,1}$ defined as follows ($\mu>0$) 
\begin{equation*}
Z_{r,\mu}=\frac{\mu(G_{r,\,1}+G_{1-r,\,1})}{G_{r,\,1}},
\end{equation*}
where $G_{r,\,1}$ and $G_{1-r,\,1}$ are independent gamma-distributed random values.

\item If $\gamma\in(0,1]$, it is the scale mixture of the tempered
Snedecor--Fisher distribution with parameters $r$ and $1$:
\begin{equation*}
M_{r,\alpha,\gamma,\lambda}\eqd \Big(\frac{S_{\gamma,1}}{\lambda
r}\cdot Q_{r,1}\Big)^{1/\alpha\gamma},
\end{equation*}
where $S_{\gamma,1}$ is a positive strictly stable random variable~\cite{Zolotarev1983}
with characteristic exponent $\gamma$ independent of the random
variable $Q_{r,1}$ with the Snedecor--Fisher distribution 
with parameters $r$ and $1$.

\item If $\gamma\in(0,1]$ and $r\in(0,1]$, it is the scale mixture
of the Pareto laws:
\begin{equation*}
M_{r,\alpha,\gamma,\lambda}\eqd
\Pi_{\alpha}\big(S_{\gamma,1}Z_{r,1}^{1/\gamma}\big)^{-1/\alpha},
\end{equation*}
where ${\mathbb P}(\Pi_{\alpha}>x)=(x^{\alpha}+1)^{-1}$, $x\ge0$.

\item If $r\in(0,1]$ and $\alpha\gamma\in(0,1]$, it is the scale
mixture of the folded normal laws:
\begin{equation*}
M_{r,\alpha,\gamma,\lambda}\eqd
|X|\cdot\frac{\sqrt{2W_1}}{\lambda^{1/\alpha\gamma}W_{\alpha}S_{\alpha\gamma,1}Z_{r,1}^{1/\alpha\gamma}},
\end{equation*}
where all the involved random variables are independent.

\end{enumerate}

\end{theorem}

\begin{proof}
To prove (i) it suffices to consider the rightmost term
in \eqref{M}, apply relations $W_1^{1/\gamma}\eqd W_{\gamma}$ and $G_{r,\mu}\eqd\dfrac{W_1}{Z_{r,\mu}}$ (here $0<r<1$ and the random variables $W_1$ and $Z_{r,\mu}$ are independent, see~\cite{Gleser1989}).

To prove (ii) it suffices to transform the rightmost term in
\eqref{M} with the account of representation~\cite{ShanbhagSreehari1977,KorolevWeibull2016} 
\begin{equation*} 
W_{\gamma}\eqd W_1\cdot S_{\gamma,1}^{-1}
\end{equation*}
(here $\gamma\in(0,1]$ and the random values on the right-hand side being independent) and use the
definition of the Snedecor--Fisher distribution as the distribution
of the ratio of two independent gamma-distributed random variables (see, e.\,g., Section~$27$ in~\cite{Johnson1995}).

To prove (iii) it suffices to transform the second term in \eqref{M}
with the account of \eqref{16} and notice that the distribution of
the ratio of two independent exponentially distributed random
variables coincides with that of the random variable $\Pi_1$.

To prove (iv) it suffices to transform the second term in \eqref{M}
with the account of \eqref{16} and notice that $W_1\eqd
|X|\sqrt{2W_1}$ with the random variables on the right-hand side
being independent (see, e.\,g.,~\cite{KorolevWeibull2016}).
\end{proof}

Such product representations for the random value $M_{r,\alpha,\gamma,\lambda}$ can be useful for its computer simulation.

\begin{theorem}
\label{Th43}
If $r\in(0,1]$, $\lambda>0$ and
$\alpha\gamma\in(0,1]$, then the d.f. $F(x;r,\alpha,\gamma,\lambda)$
is mixed exponential:
\begin{equation*}
1-F(x;r,\alpha,\gamma,\lambda)=\int_{0}^{\infty}e^{-ux}dA(u),\ \ \
x\ge0,
\end{equation*}
where $A(u)={\sf
P}\big(\lambda^{1/\alpha\gamma}W_{\alpha}S_{\alpha\gamma,1}Z_{r,1}^{1/\alpha\gamma}<u\big)$,
$u\ge0$, and all the involved random values are independent.
\end{theorem}

\begin{proof}
To prove this statement it suffices to transform the
second term in \eqref{M} with the account of \eqref{16} and obtain
\begin{equation*}
M_{r,\alpha,\gamma,\lambda}\eqd
\frac{W_1}{\lambda^{1/\alpha\gamma}W_{\alpha}S_{\alpha\gamma,1}Z_{r,1}^{1/\alpha\gamma}},
\end{equation*}
\end{proof}

\begin{theorem}
Let $r\in(0,1]$, $\alpha\gamma\in(0,1]$,
$\lambda>0$. Then the d.f. $F(x;r,\alpha,\gamma,\lambda)$ is
infinitely divisible.
\end{theorem}

\begin{proof}
This statement immediately follows from Theorem~\ref{Th43} and
the result of Goldie~\cite{Goldie1967} stating that the product of
two independent non-negative random variables is infinitely
divisible, if one of the two is exponentially distributed.
\end{proof}

It is possible to deduce explicit expressions for
the moments of the random value $M_{r,\alpha,\gamma,\lambda}$.

\begin{theorem}
\label{Th2}
Let $0<\delta<\alpha$. Then
\begin{equation*}
{\mathbb E}M_{r,\alpha,\gamma,\lambda}^{\delta}=
\frac{\Gamma\big(r+\frac{\delta}{\alpha\gamma}\big)\Gamma\big(1-\frac{\delta}{\alpha}\big)}{\lambda^{\delta/\alpha\gamma}\Gamma(r)}.
\end{equation*}
\end{theorem}

\begin{proof}
From \eqref{16} it follows that ${\mathbb E} M_{r,\alpha\gamma,\lambda}^{\delta}=\lambda^{-\delta/\alpha\gamma}{\mathbb E} G_{r,1}^{\delta/\alpha\gamma}\cdot{\mathbb E}W_1^{-\delta/\alpha}$. It
is easy to verify that ${\mathbb E} G_{r,1}^{\delta/\alpha\gamma}=\Gamma\big(r+\frac{\delta}{\alpha\gamma}\big)/\Gamma(r)$,
${\mathbb E} W_1^{-\delta/\alpha}=\Gamma\big(1-{\textstyle\frac{\delta}{\alpha}}\big)$.
Hence follows the desired result.
\end{proof}

So, the threshold method for extreme observations described in paper~\cite{Korolevetal2018} can be improved using distribution function $F(x;r,\alpha,\gamma,\lambda)$. It is worth noting that the estimation of parameters of this distribution is a rather complex computational problem.

\section{The asymptotic approximation to the distribution of the total precipitation volume during a wet period}
\label{SecApprox2}

\subsection{A statistical method for estimating stability parameters of daily precipitation trends}

Let $X_1,X_2,\ldots$ be  be the observed values of nonzero daily
precipitation volumes. The daily precipitation volumes possess the property of statistical stability~\cite{Korolevetal2018}. It has been demonstrated that there are a slight ascending trend for Elista and a slight descending one for Potsdam~\cite{Korolevetal2018}. In this section we will improve our model to obtain ``horizontal'' trends.

Suppose that for some $\beta\in(0,\infty)$
\begin{equation}\label{LLN}
\frac{1}{n^{\beta}}\sum\limits_{j=1}^nX_j\Longrightarrow
a\in(0,\infty)
\end{equation}
as $n\to\infty$. In~\cite{Korolevetal2018} we use a value of parameter $\beta$ that equals $1$. In practice, the parameter $\beta$ turns out to be close to $1$, but slightly differs from $1$ representing (slow) global trends. Let us suggest a method for statistical estimating of the parameters $a$ and $\beta$ in relation~\eqref{LLN}.

\begin{proposition}
Let $X_1,X_2,\ldots,X_n$ be the observed values of nonzero daily
precipitation volumes, $n\in\mathbb{N}$ is the total number of
available observations. For a natural $k=\overline{1,n}$ denote
$T_k=X_1+\ldots+X_k$. If condition \eqref{LLN} holds, then for $k$
large enough ($1\le m\le k\le n$) the following estimations of the parameters $a$ and $\beta$ in relation~\eqref{LLN} can be used:
\begin{gather}
\widetilde{a}=\exp{\left(\frac{\sum\limits_{k=m}^n\log T_k\cdot\sum\limits_{k=m}^n(\log
k)^2-\sum\limits_{k=m}^n\log k\cdot\sum\limits_{k=m}^n\big(\log k\cdot\log
T_k\big)}{(n-m+1)\sum\limits_{k=m}^n(\log k)^2-\big(\sum\limits_{k=m}^n\log
k\big)^2}\right)}, \label{a}\\
\widetilde\beta=\frac{\sum\limits_{k=m}^n\log T_k-(n-m+1)\widetilde{a}}{\sum\limits_{k=m}^n\log k}. \label{beta}
\end{gather}
\end{proposition}

\begin{proof}
If condition~\eqref{LLN} holds, the following approximate equality can be written:
\begin{gather}\label{appr}
\frac{T_k}{k^{\beta}}\approx a,\\
\intertext{or, equivalently,}
-\beta\log k+\log T_k\approx\log a.\notag
\end{gather}

Therefore, the estimates of the parameters $a$ and $\beta$ can be
found as the solution of the least squares problem
\begin{equation*} 
\sum_{k=m}^n(\log T_k-\beta\log k-\log a)^2\longrightarrow
\min_{\beta,\,\log a}.
\end{equation*}

This solution can be found explicitly and has the form
\begin{gather*}
\widetilde{\log a}=\frac{\sum_{k=m}^n\log T_k\cdot\sum_{k=m}^n(\log
k)^2-\sum_{k=m}^n\log k\cdot\sum_{k=m}^n\big(\log k\cdot\log
T_k\big)}{(n-m+1)\sum_{k=m}^n(\log k)^2-\big(\sum_{k=m}^n\log
k\big)^2},\\ 
\widetilde\beta=\frac{\sum_{k=m}^n\log T_k-(n-m+1)\widetilde{\log
a}}{\sum_{k=m}^n\log k}, 
\end{gather*}
that leads to the formulas~\eqref{a} and~\eqref{beta}.
\end{proof}

The choice of the auxiliary parameter $m$ is an option and depends
on, first, the time horizon over which the trend is considered and,
second, on that the accuracy in~\eqref{appr} should be sufficient.
The greater $m$, the more ``global'' is the trend under consideration. It should be emphasized that in~\eqref{LLN} we {\it do not} assume
that $X_1,X_2,\ldots$ are independent.

Fig.~\ref{Averages} illustrates the performance of the least
squares algorithm described above with $m=3000$. On this figure the
values of the parameter $a$ are estimated as $4.087$ for Potsdam and $0.96$
for Elista using formula~\eqref{a}, whereas the obtained by~\eqref{beta} values of $\beta$ appear to be equal to $1.139$ for Elista and $0.981$ for
Potsdam, respectively. So, the corresponding trends are horizontal.

\begin{figure}[h]
\centering 
\includegraphics[width=\textwidth,
height=0.35\textheight]{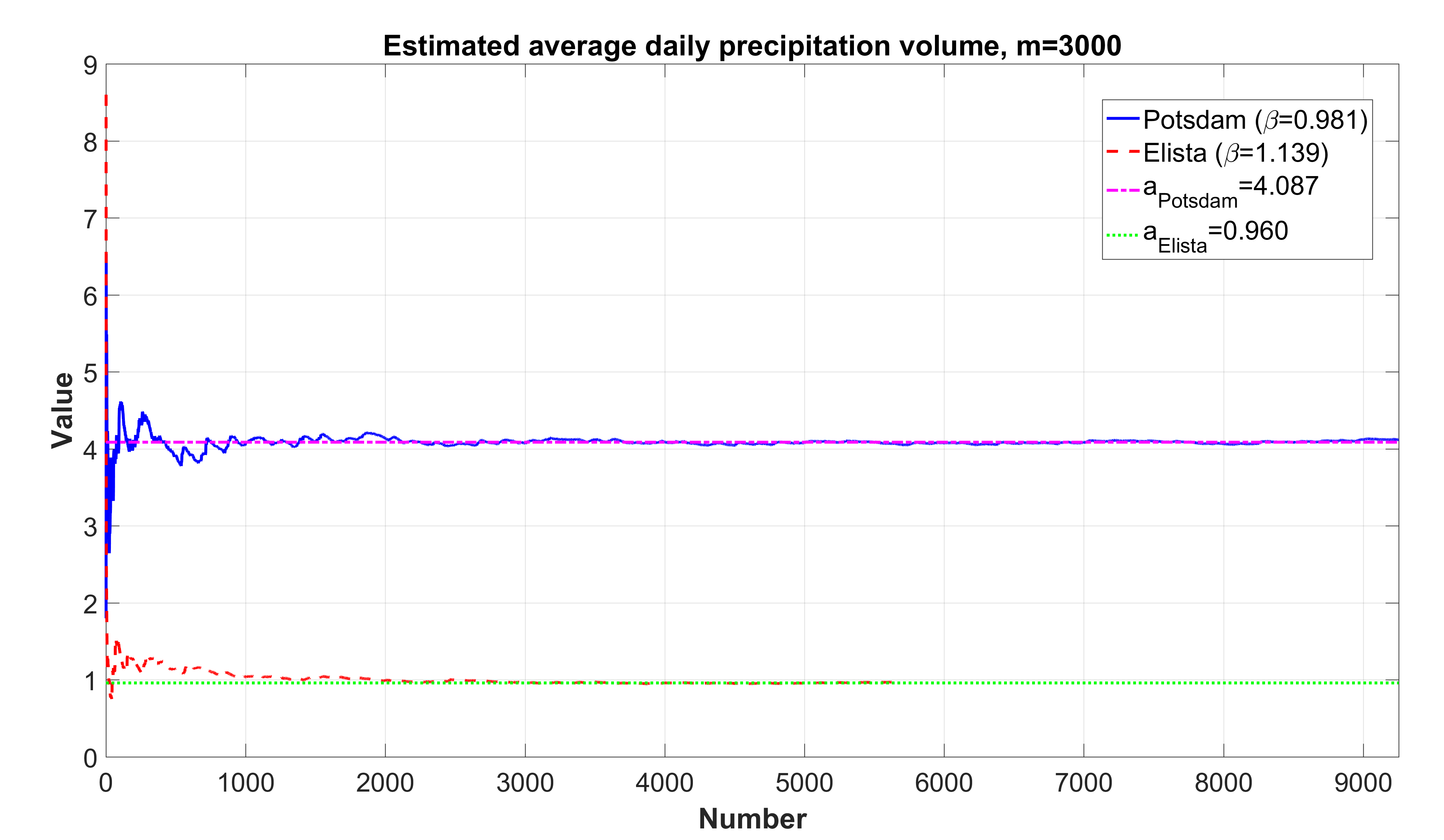}
\caption{Stabilization of the cumulative
averages of daily precipitation volumes as $n$ grows 
with $\beta=1.139$ for Potsdam (solid line) and with
$\beta=0.981$ for Elista (dashed line).}\label{Averages}
\end{figure}

\subsection{Generalization of the R{\'e}nyi theorem for GNB random sums}

The following theorem is a generalization of the R{\'e}nyi theorem
that dealt with rarefied renewal processes~\cite{Kalashnikov1997}. Instead of geometric sums of independent identically distributed random variables acting in the R{\'e}nyi
theorem, here we will consider the following analog of the law of
large numbers for GNB random sums of not necessarily independent and
not necessarily identically distributed random variables.

\begin{theorem}
\label{Th3}
Assume that the daily precipitation volumes on
wet days $X_1,X_2,...$ satisfy condition \eqref{LLN} with some
$\beta>0$ and $a>0$. Let the numbers $r>0$, $\alpha$ and $\lambda>0$
be arbitrary. For each $n\in\mathbb{N}$, let the random value
$N_{r,\alpha,\mu_n}$ have the GNB distribution with parameters $r$,
$\alpha$ and $\mu_n=\lambda/n^{\alpha}$. Assume that the random values
$N_{r,\alpha,\mu_n}$ are independent of the sequence $X_1,X_2,...$
Then
\begin{equation*}
\frac{a\lambda^{\beta/\alpha}}{n^{\beta}}\sum\limits_{j=1}^{N_{r,\alpha,\lambda/n^{\alpha}}}X_j\Longrightarrow
\overline{G}_{r,\alpha/\beta,1}\eqd G_{r,1}^{\beta/\alpha}
\end{equation*}
as $n\to\infty$. 
\end{theorem}

\begin{proof} 
The proof is based on Lemma~\ref{Lemma4} and \eqref{2}. From
\eqref{2} it follows that
\begin{equation}\label{111}
\frac{\lambda^{1/\alpha}}{n}\cdot
N_{r,\alpha,\lambda/n^{\alpha}}\Longrightarrow
\overline{G}_{r,\alpha,1}
\end{equation}
as $n\to\infty$. By virtue of condition \eqref{LLN}, in Lemma~\ref{Lemma4}  let
$b_n=n^{\beta}/a$. As $N_n$ in Lemma~\ref{Lemma4} take
$N_{r,\alpha,\lambda/n^{\alpha}}$. Then
$b_{N_n}=\frac{1}{a}N_{r,\alpha,\lambda/n^{\alpha}}^{\beta}$. From
\eqref{111} it follows that, as $n\to\infty$,
\begin{equation}\label{112}
{\textstyle\frac{1}{a}}N_{r,\alpha,\lambda/n^{\alpha}}^{\beta}\cdot\frac{c^{\beta/\alpha}}{n^{\beta}}\Longrightarrow
{\textstyle\frac{1}{a}}\overline{G}_{r,\alpha,1}^{\beta}\eqd{\textstyle\frac{1}{a}}\overline{G}_{r,\alpha/\beta,1}\eqd
{\textstyle\frac{1}{a}}G_{r,1}^{\beta/\alpha}.
\end{equation}
Therefore, as $d_n$ we can take
$d_n=n^{\beta}/\lambda^{\beta/\alpha}$. So, using \eqref{112} in the
role of \eqref{T31} in Lemma~\ref{Lemma4}, we obtain \eqref{T32} in the form
\begin{equation}\label{113}
\frac{\lambda^{\beta/\alpha}}{n^{\beta}}\sum\limits_{j=1}^{N_{r,\alpha,\lambda/n^{\alpha}}}X_j\Longrightarrow
{\textstyle\frac{1}{a}}\overline{G}_{r,\alpha/\beta,1}\eqd
{\textstyle\frac{1}{a}}G_{r,1}^{\beta/\alpha},
\end{equation}
whence follows the desired result. 
\end{proof} 

Theorem~\ref{Th3} presents a good tool for the account of the parameters
$\beta$ and $\alpha$ characterizing the deviation from traditional
NB and arithmetic mean models due to the influence of possible
(slow) global trends. If in Theorem~\ref{Th3} $r=\alpha=\beta=1$, then we obtain a version of the R{\'e}nyi theorem~\cite{Kalashnikov1997} generalized to
non-identically distributed and not necessarily independent
summands. If in Theorem~\ref{Th3} $\alpha=1$, then we obtain the law of
large numbers for negative binomial random sums~\cite{BevraniKorolev2016}.

Therefore, if daily precipitation volumes are considered as
$X_1,X_2,\ldots$ (of course, being non-identically distributed and
not independent), with the account of the excellent fit of the GNB
model for the duration of a wet period (see Fig.~\ref{FigGNBl1Potsdam}), with rather
small $\mu$, the GG distribution can be regarded as an adequate and
theoretically well-based model for the total precipitation volume
per (long enough) wet period.

\section{Comparison of extremality tests based on assumptions about gamma and generalized gamma distributions of precipitation volumes}

\label{SecComparison}

\subsection{The test for data abnormality based on the GG distribution}

Let $m\in\mathbb{N}$ and
$\overline{G}^{(1)}_{r,\gamma,\mu},\overline{G}^{(2)}_{r,\gamma,\mu},\ldots,\overline{G}^{(m)}_{r,\gamma,\mu}$
be independent random values having the same GG distribution with
parameters $r>0$, $\gamma$ and $\mu>0$. Also, let
$G^{(1)}_{r,\mu},G^{(2)}_{r,\mu},\ldots,G^{(m)}_{r,\mu}$ be
independent random values having the same gamma distribution with
parameters $r>0$ and $\mu>0$.

The base for the first step in the construction of the desired test
is the following obvious conclusion: if the random values
$\overline{G}^{(1)}_{r,\gamma,\mu},\overline{G}^{(2)}_{r,\gamma,\mu},\ldots,\overline{G}^{(m)}_{r,\gamma,\mu}$
are identically distributed (that is, the sample
$\overline{G}^{(1)}_{r,\gamma,\mu},\overline{G}^{(2)}_{r,\gamma,\mu},\ldots,\overline{G}^{(m)}_{r,\gamma,\mu}$
is homogeneous), then the random values
$\big(\overline{G}^{(1)}_{r,\gamma,\mu}\big)^{\gamma},\big(\overline{G}^{(2)}_{r,\gamma,\mu}\big)^{\gamma},\ldots,\big(\overline{G}^{(m)}_{r,\gamma,\mu}\big)^{\gamma}$
are also identically distributed (that is, the sample
$\big(\overline{G}^{(1)}_{r,\gamma,\mu}\big)^{\gamma},\big(\overline{G}^{(2)}_{r,\gamma,\mu}\big)^{\gamma},\ldots,\big(\overline{G}^{(m)}_{r,\gamma,\mu}\big)^{\gamma}$
is homogeneous. Consider the following relative contribution of the random value
$\big(\overline{G}^{(1)}_{r,\gamma,\mu}\big)^{\gamma}$ to the sum
$\big(\overline{G}^{(1)}_{r,\gamma,\mu}\big)^{\gamma}+\big(\overline{G}^{(2)}_{r,\gamma,\mu}\big)^{\gamma}+\ldots+\big(\overline{G}^{(m)}_{r,\gamma,\mu}\big)^{\gamma}$:
\begin{equation}
\label{R}
R=\frac{(m-1)\big(\overline{G}^{(1)}_{r,\gamma,\mu}\big)^{\gamma}}{\big(\overline{G}^{(2)}_{r,\gamma,\mu}\big)^{\gamma}+
\ldots+\big(\overline{G}^{(m)}_{r,\gamma,\mu}\big)^{\gamma}}
\eqd
Q_{r,k}.
\end{equation}

Here, relation $(\overline{G}_{r,\gamma,\mu})^{\gamma}\eqd G_{r,\mu}$ is used.
Therefore, the statistical approach introduced in~\cite{Korolevetal2018} can be applied. The homogeneity test for a sample from the GG distribution
is based on the random value $R$~\eqref{R} that has the Snedecor--Fisher
distribution with parameters $r$ and $k=(m-1)r$.

\begin{proposition}
Let $V_1,\ldots,V_m$ ($V_1\ge V_j$ for all
$j\ge2$) be the total precipitation
volumes during $m$ wet periods. 
Under the hypothesis $H_0$ (``the precipitation volume
$V_1$ under consideration is not abnormally large'') the random value
\begin{equation}
\label{SR_GG}
SR_{GG}=\frac{(m-1)V_1^{\gamma}}{V_2^{\gamma}+\ldots+V_m^{\gamma}}.
\end{equation}
has the Snedecor--Fisher distribution with parameters $r$ and
$k=(m-1)r$.

Let $q_{r,k}(1-\alpha)$ be the $(1-\alpha)$-quantile of the
corresponding Snedecor--Fisher distribution ($\alpha\in(0,1)$ is a small number). If $SR_{GG}>q_{r,k}(1-\alpha)$, then the hypothesis $H_0$ must be
rejected.
\end{proposition}

If the  hypothesis $H_0$ is rejected, the volume $V_1$ of precipitation during one wet
period must be regarded as abnormally large. The probability of erroneous rejection of $H_0$ is equal to $\alpha$.

\subsection{Comparison of statistical tests using real data}

In this section we present the results of the application of the
test based on the quantity $SR_{GG}$~\eqref{SR_GG} to the analysis of the time series of daily precipitation observed in Potsdam and Elista from $1950$ to $2008$. Moreover, we compare potentially extreme values obtained by GG-test based on the statistic  $SR_{GG}$ with test for gamma random values~\cite{Korolevetal2018}. It uses the statistic $SR$ that matches with the quantity $SR_{GG}$, if the parameter $\gamma$ in~\eqref{SR_GG} equals $1$.

The results of the application of the tests for a total precipitation
volume during one wet period to be abnormally large based on $SR_{GG}$ and $SR$
in the moving mode~\cite{Korolevetal2018} are shown on Figs.~\ref{FigPotsdamExtrTest360}
(Potsdam) and~\ref{FigElistaExtrTest360} (Elista). A fixed sample point can be one of the following types ($m$ is a size of window):
\begin{itemize}
\item absolutely extreme (if all $m$ windows contain this observation);
\item intermediate (if more than half windows contain it);
\item relatively (if at least one window contains it);
\item not extreme.
\end{itemize}

For the sake of vividness on these figures the time horizon equals $360$ days and the significance level $\alpha$ of the tests is $0.01$. The absolutely, intermediate and relatively abnormal precipitation volumes are marked with downward-pointing triangles, circles and squares, respectively, for test based on the statistic $SR$, whereas corresponding ones for the statistic $SR_{GG}$ based on the GG distribution are presented with upward-pointing triangles, diamonds and right-pointing triangles, respectively. It is worth noting that MATLAB's notations are used here for these markers due to our implementation of corresponding computational procedures.

\begin{figure}[!h]
\centering
\includegraphics[width=\textwidth,
height=0.3\textheight]{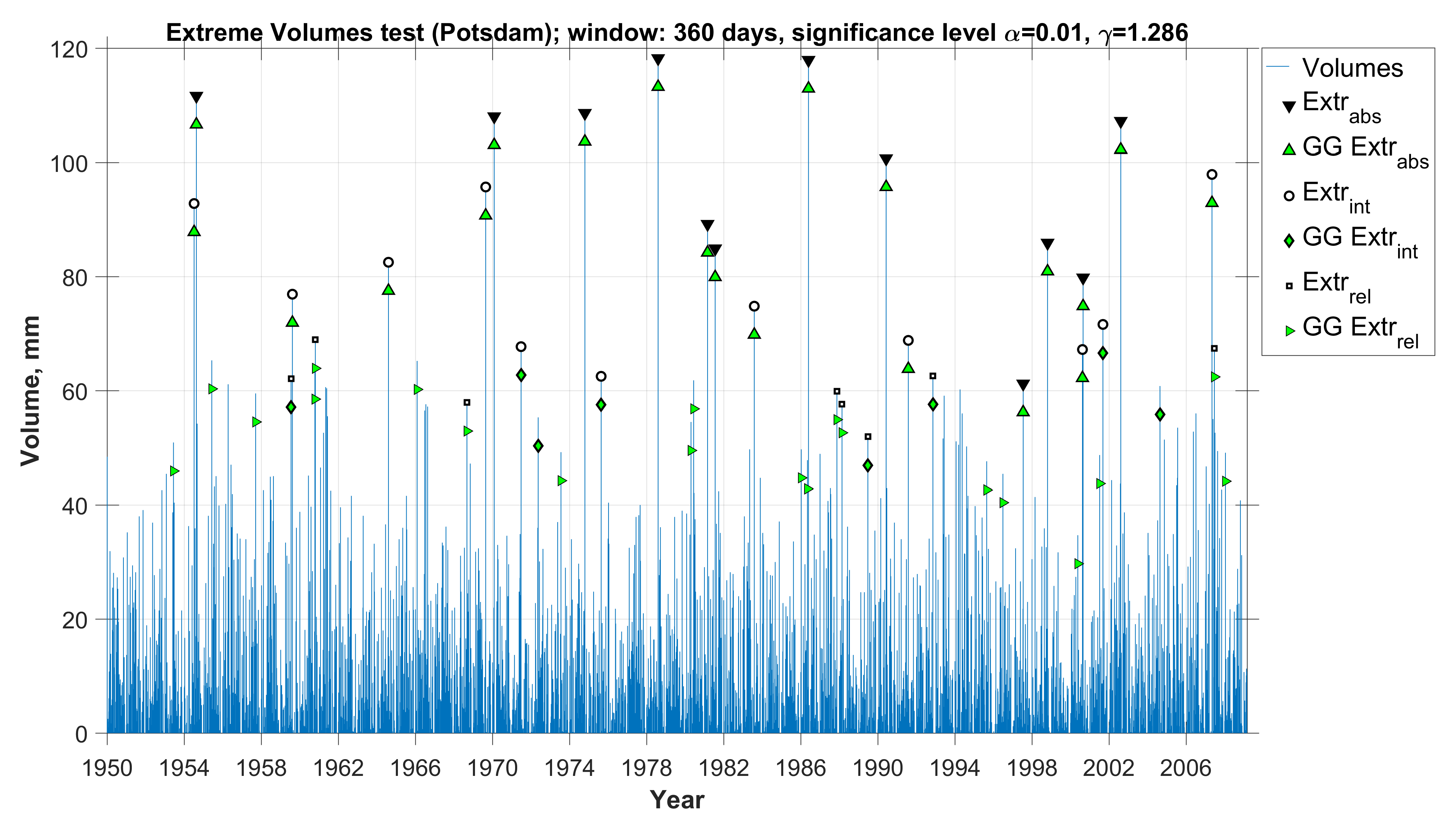}
\caption{Absolutely, intermediate and relatively abnormal precipitation volumes (Potsdam). 
}\label{FigPotsdamExtrTest360}
\end{figure}

The Figs.~\ref{FigPotsdamExtrTest360} and~\ref{FigElistaExtrTest360} demonstrate non-trivial values of parameter $\gamma$, that is, $\gamma\neq1$. For Potsdam $\gamma=1.286$, whereas for Elista $\gamma$ equals $1.279$. At the same time, the results of the two methods are quite close, although the approach based on the GG distribution demonstrates a higher quality of determining potentially extreme observations. The same conclusions are valid for smaller window sizes, for example, $90$ days.

\begin{figure}[!h]
\centering
\includegraphics[width=\textwidth,
height=0.3\textheight]{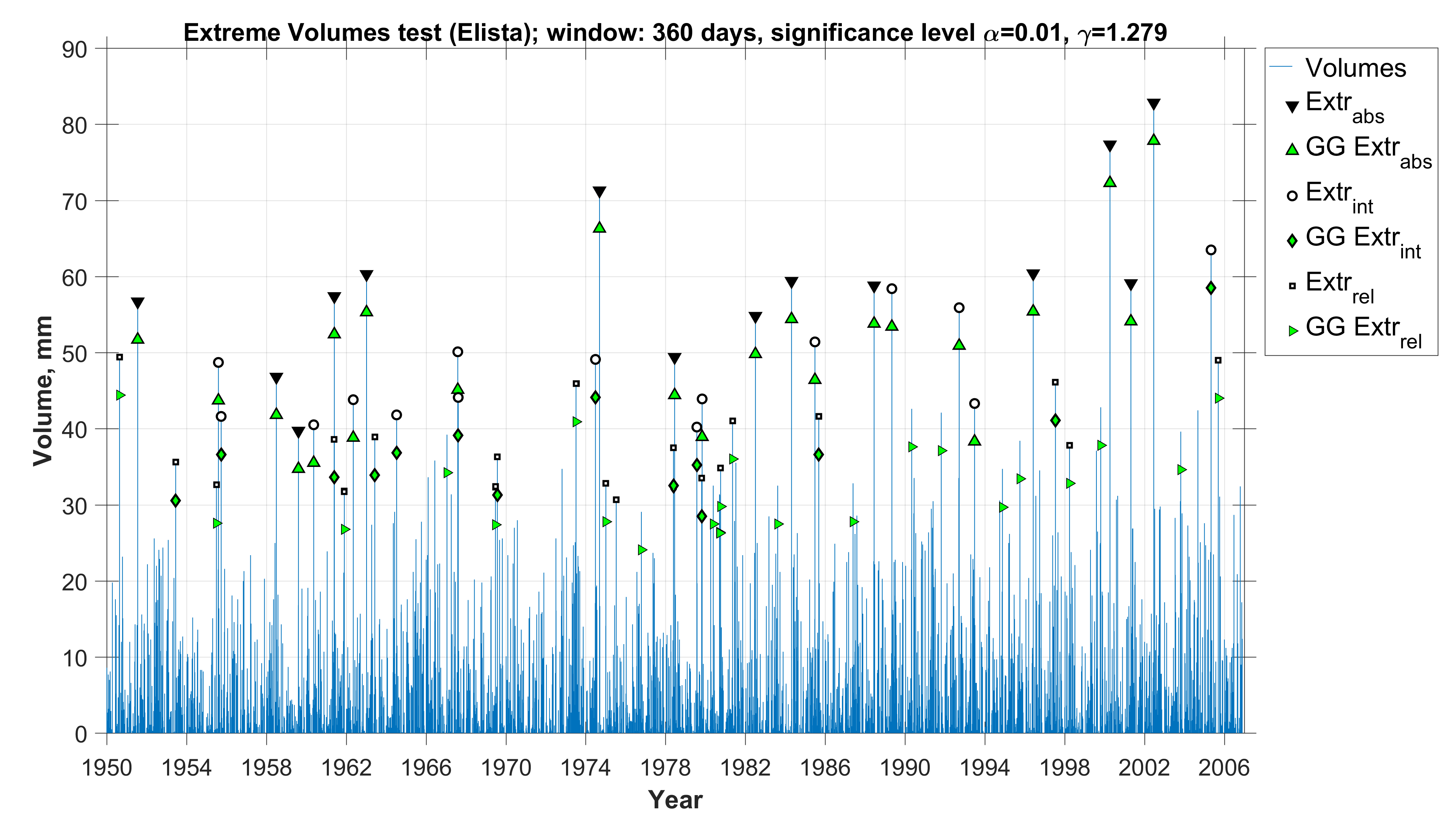}
\caption{Absolutely, intermediate and relatively abnormal precipitation volumes (Elista).
}\label{FigElistaExtrTest360} 
\end{figure}

\section{Conclusions and discussion}
\label{SecConclusion}

The article has considered asymptotic models for some precipitation characteristics based on GNB distributions. Also, a statistical test based on GG distribution to determine the type of precipitation extremes has proposed. These distributions are not quately widespread, so the methods for estimating their parameters are often not implemented in standard statistical packages. Therefore, the implementation of appropriate procedures requires the creation of specialized software solutions, for example, based on the functional approach, as was done in this article. However, as demonstrated in the article, the results of fitting such distributions to real data has turned out to be better compared with classical ones. Therefore, for processing spatial meteorological data from a large number of stations, the proposed methods and models can be effectively implemented as services using high-performance computing.

\section*{Acknowledgment}
The research was partially supported by the Russian Foundation for
Basic Research (project~{\color{blue}17-07-00851}) and the RF Presidential
scholarship program (No.~{\color{blue}538.2018.5}).

\end{document}